\documentclass[12pt]{elsarticle}

\makeatletter
\def\ps@pprintTitle{%
 \let\@oddhead\@empty
 \let\@evenhead\@empty
 \def\@oddfoot{\centerline{\thepage}}%
 \let\@evenfoot\@oddfoot}
\makeatother

\AtBeginDocument{
  \let\div\relax
  \DeclareMathOperator{\div}{div}
}
\def\R{{\mathbb R}}

\def\dee{{\rm d}}
\def\e{{\rm e}}
\def\:{\colon}

\usepackage{graphicx,epsfig,psfrag}
\usepackage{amssymb}
\usepackage{amsthm}
\usepackage{mathrsfs}
\usepackage{calrsfs}
\usepackage{microtype}
\usepackage{mathscinet}
\usepackage{hyperref}
\usepackage{amsmath}
\usepackage{enumerate}
\newtheorem{theorem}{Theorem}[section]

\newtheorem{corollary}[theorem]{Corollary}
\newtheorem{lemma}[theorem]{Lemma}

\newtheorem{definition}[theorem]{Definition}
\newtheorem{remark}[theorem]{Remark}
\def\dist{{\rm dist}}

\newcommand{\be}{\begin{equation}}
\newcommand{\ee}{\end{equation}}
\newcommand{\ba}{\begin{eqnarray}}
\newcommand{\bs}{\begin{eqnarray*}}
\newcommand{\ea}{\end{eqnarray}}
\newcommand{\es}{\end{eqnarray*}}
\newcommand{\bi}{\begin{itemize}}
\newcommand{\ei}{\end{itemize}}
\newcommand{\A}{{\alpha}}
\newcommand{\eps}{{\epsilon}}

\newcommand{\p}{{\partial}}
\newcommand{\Om}{{\Omega}}

\newcommand{\PO}{{\cal{P}}}
\newcommand{\BC}{{\cal{B}}}
\newcommand{\LO}{{\cal{L}}}
\newcommand{\B}{{\beta}}

\newcommand{\Omb}{\overline{\Om}}

\begin{document}

\begin{frontmatter}

\author[RL]{R.\ Laister\corref{thing}}
\cortext[thing]{Corresponding author}
\ead{Robert.Laister@uwe.ac.uk}
\address[RL]{Department of Engineering Design and Mathematics,  University of the West of England,\\ Bristol BS16 1QY, UK.}
\author[warwick]{J.C.\ Robinson}
\ead{J.C.Robinson@warwick.ac.uk}
\address[warwick]{Mathematics Institute, Zeeman Building, University of Warwick,\\ Coventry CV4 7AL, UK.}
\author[warsaw]{M.\ Sier{\.z}\polhk{e}ga}
\ead{M.Sierzega@mimuw.edu.pl}
\address[warsaw]{Faculty of Mathematics, Informatics and Mechanics,  University of Warsaw,\\  Banacha 2, 02-097 Warsaw, Poland.}


\title{A Necessary and Sufficient Condition for Uniqueness of the Trivial Solution in  Semilinear Parabolic Equations}

\begin{abstract}
 In their (1968) paper  Fujita and Watanabe   considered the issue of uniqueness of the trivial solution  of semilinear parabolic equations with respect to the class of bounded, non-negative solutions. In particular they showed that if the underlying ODE has non-unique solutions (as characterised via an Osgood-type condition) {\em and} the  nonlinearity $f$ satisfies  a concavity condition, then the parabolic PDE also inherits the non-uniqueness property. This concavity assumption has remained in place either implicitly or explicitly in all subsequent work in the literature relating to this and other, similar,  non-uniqueness phenomena in parabolic equations. In this paper we provide an elementary proof of non-uniqueness for the PDE without any such concavity assumption on $f$. An important consequence of our result is that uniqueness of the trivial solution  of the PDE is equivalent to uniqueness of the trivial solution  of the corresponding ODE, which in turn is known to be equivalent to an Osgood-type integral condition on $f$.
\end{abstract}

\end{frontmatter}

\section{Introduction}

We consider the issue of uniqueness (with respect to bounded solutions) of the trivial solution of the semilinear parabolic  problem
\bs
(\text{P})\quad\left\{
\begin{array}{rlll}
u_t  & = \LO u+ q(x)f(u) & \text{in} &    Q_T:=\Om\times (0,T),\\
\BC u & =  0  & \text{in}  & \partial\Om\times (0,T],\\
u(x,0) & =  0  & \text{in}  & \Omb ,
\end{array}
\right.
\es
where $\LO$ is a uniformly elliptic operator,  $\BC$ is a boundary operator, $q(x)\ge 0$  and   $f(0)=0$. The domain $\Omega\subset\R^d$ ($d\ge 1$) is  bounded with  boundary $\p\Om$ of class $C^{2+\A}$, so that classical parabolic regularity and maximum principles apply. Without loss of generality we also assume that $\Om$  contains the origin. As in Fujita and Watanabe \cite{FW} we take $\LO$  of the form
\be
\LO u=\sum_{i,j=1}^d a_{ij}(x){u}_{x_ix_j}+\sum_{j=1}^d  b_{j}(x){u}_{x_j}+c(x)u
\label{L}
\ee
where $a_{ij}$ is symmetric and satisfies the uniform ellipticity condition
\be
k |y|^2\le \sum_{i,j=1}^d a_{ij}(x)y_iy_j\le |y|^2/k,\qquad \forall x\in\Om ,\ \forall y\in\R^d
\label{unifell}
\ee
for some $k>0$. The boundary operator in (P) is given by
\be
\BC u:= \B (x) \frac{\p u}{\p\nu} +(1-\B (x))u,\label{B}
\ee
where  $0\le \B(x)\le 1$ and $\p u/\p\nu$ is the  conormal derivative
\be
\frac{\p u}{\p\nu}(x)= \sum_{i,j=1}^d {u}_{x_i}(x)a_{ij}(x)n_j(x) \label{conormal}
\ee
with $n(x)=(n_1(x),n_2(x),\ldots ,n_d(x))$ being the unit outer normal at $x\in\p\Om$. The case $\B\equiv 1$ therefore corresponds to Neumann boundary conditions, whilst $\B\equiv 0$ corresponds to  Dirichlet boundary conditions. Other choices of  $\B (x)$  represent  Robin or mixed boundary conditions, which in a certain sense (regarding the ordering of the corresponding heat kernels) is intermediate between the Neumann and Dirichlet cases. We provide precise regularity conditions on the coefficients $\LO$ and $\BC$ in a later section.

Since $f(0)=0$ and the initial data in (P) is zero,  $u=0$ is a solution of both the PDE problem (P) and the ODE problem
 \be
 \dot{u}=f(u),\qquad u(0)=0.\label{ODE}
  \ee
The nonlinearity $f$ is assumed continuous, non-decreasing and positive for $u>0$. For such $f$ it is well known (\cite{Osg}) that $u=0$ is the unique local solution of (\ref{ODE}) if and only if the following Osgood integral condition holds:
\be
\int_0^\eps \frac{\dee u}{f(u)}  =\infty \qquad \text{for some } \eps >0.\label{osgood}
\ee
Our main result is that if the integral condition (\ref{osgood}) does not hold then problem (P) possesses non-unique, non-negative, bounded solutions; see Theorem~\ref{thm:possoln}. An important consequence of our result is that uniqueness of the trivial solution of the PDE (P) is equivalent to uniqueness of the trivial solution  of the ODE (\ref{ODE}). Thus uniqueness of the  trivial solution of  (P)  is equivalent to (\ref{osgood}); see Corollary~\ref{cor:main}.

This problem was considered almost half a century ago by  Fujita and Watanabe
 \cite{FW} (see also \cite{Fuj}). They proved \cite[Theorem 1.4]{FW} that the Osgood condition (\ref{osgood}) is sufficient for uniqueness of the trivial solution in (P) but did not prove  necessity under the same conditions. In order to establish non-uniqueness  when  condition (\ref{osgood}) fails, the authors  imposed an additional concavity assumption on the nonlinearity $f$ \cite[Theorem 1.5]{FW}. We show here that this concavity assumption is not required and thereby obtain a result valid for any increasing function $f$. Given the many works in the literature which have utilised and extended Fujita and Watanabe's non-uniqueness result it is surprising that their concavity assumption has remained until now. We suspect that this may be due to the enthusiasm  for studying  the `model' nonlinearity $f(u)=u^p$, for $0<p<1$, which is of course concave.

There have been several papers subsequent to \cite{FW} providing  non-uniqueness results for parabolic equations of various types,
e.g. with unbounded coefficients \cite{Loa},  degenerate $p$-Laplacian operators \cite{BGKT,BBGKT,HW}, and systems \cite{Bok,DE1,EH1,HW}. However,  all  these works either  assume explicitly that $f$ is concave near zero \cite{CDE,Fuj} or implicitly by working only with nonlinearities of  power law type,  $f(u)=u^p$ ($0<p<1$). To the best of our knowledge non-uniqueness has not been established without assuming concavity of $f$. We remark that non-uniqueness with respect to  {\em unbounded} solutions can also occur in parabolic equations even when the corresponding ODE has unique solutions (e.g. when $f$ is locally Lipschitz) \cite{HW,MT,NS}. However the  non-uniqueness  there is due to the superlinear growth of $f$ at infinity.

\section{Non-Uniqueness of Bounded Solutions}

We state our assumptions on the problem data:
\bi
\item[(H1)] The  coefficients of  $\LO$ in (\ref{L}) satisfy  $a_{ij}\in C^{2+\A}(\Omb )$, $b_j\in C^{1+\A}(\Omb )$, $c\in C^{\A}(\Omb )$ and the uniform ellipticity condition (\ref{unifell}).
\item[(H2)] The coefficient of $\BC$ in  (\ref{B}) satisfies   $0\le \B(x)\le 1$ and  $\B\in  C^{2+\A}(\p\Om )$.
\item[(H3)] $q\ge 0$, $q\not\equiv 0$ and $q\in C(\Omb )$,
\item[(H4)] $f:[0,\infty )\to [0,\infty )$ is continuous, non-decreasing, $f(0)=0$ and $f>0$ on $(0,\infty )$.
\ei

In all that follows $S_{\B}(t):L^{\infty}(\Om )\to L^{\infty}(\Om )$ ($t\ge 0$) denotes the  semigroup generated by $-\LO$ with   boundary conditions $\BC u=0$. It is well-known (e.g. \cite{Aro,Fri}) that one has the  representation formula
\be
[S_{\B}(t)\psi ](x)= \int_{\Om} K_{\B}(x,y;t)\psi (y)\, \dee y,\qquad \psi\in L^{\infty}(\Om ),\label{eq:semi}
\ee
where $K_{\B}$ is the  kernel (synonymously known  as the fundamental solution or parabolic Green's function) associated with $\LO$ with the same boundary conditions. For notational convenience we henceforth  write $S_{D}$ and $K_D$  in the Dirichlet case  $\B\equiv 0$. The open Euclidean ball in $\R^d$, centred at $x$, with radius $R$ will be denoted  $B_R(x)$ and   $\chi_{R}$ denotes the characteristic function on  $B_R:=B_{R}(0)$.

\begin{definition}{
We say that $u$ is a \emph{bounded generalised solution} of {\rm (P)} on $[0,T]$ if $u\in L^{\infty}(Q_T )$, $u\ge 0$ and satisfies
\be
u(x,t)=\int_0^t \int_{\Om} K_{\B}(x,y;t-s)[q(y)f(u(y,s))]\, \dee y\dee  s,\label{eq:integral}
\ee
or equivalently,
\begin{equation}
u(t)=\int_0^t S_{\B}(t-s)[qf(u(s))]\, \dee  s.\label{eq:bounded}
\end{equation}}
\label{def:soln}
\end{definition}

\begin{remark}
If $u$ is a bounded generalised  solution  and $q$ and $f$ are H\"{o}lder continuous on $\Om$ and $(0,M ]$ respectively, where $M=\|u\|_{\infty}$, then $u$ is
a solution of {\rm (P)} in the classical sense by standard parabolic regularity results of De Giorgi-Nash-Moser type and classical Schauder estimates (cf. \cite[Appendix B]{QS}).

\label{rem:classical}
\end{remark}

Clearly $u\equiv 0$ is a classical solution of (P) on $[0,T]$ for any $T>0$.  We will require the following comparison result from \cite{FW}, reformulated
 here for the reader's convenience.

\begin{lemma} \cite[Lemma 2.7]{FW}.
If {\rm (H1)-(H2)} hold  then $ K_{\B}(x,y;t)\ge  K_D(x,y;t)$ for all $x,y\in\Om$, $t>0$. Consequently the corresponding semigroups satisfy
$S_{\B}(t)\psi \ge S_{D}(t)\psi$ for all non-negative $\psi\in L^\infty(\Om )$, $t> 0$.
\label{lem:Sorder}
\end{lemma}

Our goal is to show that if the Osgood condition (\ref{osgood}) fails then there  exists a non-trivial  subsolution of (P). To achieve this we utilise
a Gaussian lower bound on the Dirichlet  kernel $K_D$ due to Aronson \cite{Aro}. There the author considered more general linear parabolic operators of the form
\be
\PO u=u_t-\sum_{i,j=1}^d \left(A_{ij}(x,t){u}_{x_i}\right)_{x_j}
-\sum_{j=1}^d  \left(A_{j}(x,t)u\right)_{x_j}
-\sum_{j=1}^d  B_{j}(x,t){u}_{x_j}
-C(x,t)u\label{eq:PO}
\ee
under fairly minimal regularity assumptions on the  coefficients. In the special case where
\ba
A_{ij}(x,t) & = & a_{ij}(x), \quad   A_j(x,t)=0, \quad  B_j(x,t)=b_j(x)-\sum_{i=1}^d  \left( a_{ij}(x)\right)_{x_i}, \label{coeff1}\\
C(x,t) & = & c(x)-\sum_{j=1}^d
 \left( b_{j}(x)\right)_{x_j}-\sum_{i,j=1}^d \left( a_{ij}(x)\right)_{x_ix_j}\label{coeff2}
\ea
then the operator $\PO$ reduces to that of $\p /\p t - \LO$ in problem (P). In particular, if assumption (H1) holds then the coefficients
of $\PO$ given by (\ref{coeff1}-\ref{coeff2}) are all H\"{o}lder continuous and thus certainly bounded. This, together with (\ref{unifell}), ensure that the
results in \cite{Aro} are applicable to (P) with Dirichlet boundary conditions. For similar estimates in the special case where $\LO$ is the Laplacian
operator with homogeneous Dirichlet or Neumann boundary conditions see \cite{vdB89,vdB90} for a more concise treatment.

\begin{lemma} Assume {\rm (H1)} holds and $c(x)\le 0$. Let $r\in (0,1)$ be such that   $B_{3r}\subset\Om$ and let $\delta =\dist (B_{2r},\p\Om )>0$.
Then there exists a constant $\kappa =\kappa (d, \LO ,\delta )>0$   such that
\be
S_{D}(t)\chi_{r}\ge \kappa \chi_{r},\qquad \forall t\in (0,r^2/8].\label{Sbound1}
\ee
\label{lem:SD}
\end{lemma}

\begin{proof}
By  \cite[Theorem 8, Theorem 9 (iii)]{Aro} (with $\Om^\prime =B_{2r}$, $T=1$ and $\tau =0$  in the notation of that paper) we have the following lower bound on the heat kernel $K_{D}$ associated with the operator $\LO$ with Dirichlet boundary conditions:
\bs
K_{D}(x,y;t)\ge c_1t^{-d/2}\e^{-c_2|x-y|^2/t}
\es
for all $x,y\in B_{2r}$ and $0<t\le\min\{1,\dist^2(y,\p B_{2r})/8\}$, where $c_1$ and $c_2$ are positive constants depending only upon $d$, $\delta$ and $\LO$. In particular, for $y\in B_{r}$ we have  $\dist(y,\p B_{2r})\ge r$ and so
\bs
K_{D}(x,y;t)\ge c_1t^{-d/2}\e^{-c_2|x-y|^2/t}
\es
for all $x,y\in B_{r}$ and $0<t\le\min\{1,r^2/8\}=r^2/8$. Hence for all such $t$,
\bs
[S_{D}(t)\chi_r](x)=\int_{B_r}K_{D}(x,y;t)\,\dee y\ge c_1t^{-d/2}\int_{B_r}\e^{-c_2|x-y|^2/t}\,\dee y.
\es
The latter integral is simply a constant multiple of the representation of the solution of a heat equation of the form $u_t=C \Delta u$ on the whole space $\R^d$ with the radially symmetric, non-increasing initial data $\chi_{r}(x)$.  Consequently this integral is also radially symmetric and decreasing with $|x|$ and so for $|x|\le r$, choosing any unit vector $\hat{u}$, we can write
\bs
[S_{D}(t)\chi_r](x)\ge c_1t^{-d/2}\int_{B_r (r{\hat{u}})}\e^{-c_2|z|^2/t}\,\dee z
=c_1\int_{B_{r/\sqrt t}(\frac{r}{\sqrt{t}}{\hat{u}})}\e^{-c_2|w|^2}\,\dee w.
\es
Observing that for $r/\sqrt t\ge 1$ we have
$$
B_{r/\sqrt t}({\textstyle\frac{r}{\sqrt{t}}}{\hat{u}})\supseteq B_{1}({\hat{u}})
$$
it follows that
\bs
[S_{D}(t)\chi_r](x) \ge  c_1\int_{B_{1}({\hat{u}})}\e^{-c_2|w|^2}\,\dee w =:\kappa^\prime =\kappa^\prime \chi_r(x)
\es
for all $x\in B_{r}$ and $0<t\le \min\{r^2/8,r^2\}=r^2/8$.

Clearly for $x\not\in B_{r}$ we have $[S_{D}(t)\chi_{r}](x)\ge 0= \chi_{r}(x)$ and so the result
follows with $\kappa = \min\{ 1,\kappa^\prime\}$.
\end{proof}

Lemma~\ref{lem:SD} is central to our proof of non-uniqueness. Although elementary, similar versions have proved extremely powerful in establishing fundamental non-existence results for semilinear heat equations in Lesbesgue spaces. For example, a version was used in \cite{LRSV} to give a complete characterisation of those $f$ for which the local existence property holds, and another in \cite{LRS}  to establish instantaneous blow-up for singular initial data even when all solutions of the corresponding ODE exist globally in time.

We can now prove our main result.

\begin{theorem}
Assume {\rm (H1)-(H4)} hold. If $f$ does not satisfy the Osgood condition (\ref{osgood}) (i.e., $\int_0^\eps \dee u/f(u)  <\infty$)  then there exists a $T>0$ and a non-trivial, bounded generalised  solution $U$ of  {\rm (P)} on $[0,T]$ satisfying $U(x,t)>0$ on $Q_T$. Furthermore, if $q\in C^\A (\Om )$ and
$f\in C^\A ((0,M])$ (where $M=\|U\|_{\infty}$) for some $\A\in (0,1)$ then $U$ is a classical solution of {\rm (P)}.
\label{thm:possoln}
\end{theorem}

\begin{proof}
Suppose initially that  $c(x)\le 0$. By (H3) there exist $\rho >0$, $\gamma >0$ and $x_0\in\Om$ such that $B_{3\rho} (x_0)\subset\Om$ and $q(x)\ge \gamma$ for all $x\in B_{3\rho} (x_0)$. Without loss of any generality we may assume that $x_0$ is the origin. Now choose $r$ as in Lemma \ref{lem:SD}, so that
\be
S_{D}(t)\chi_{r}\ge  \kappa \chi_{r},\qquad \forall t\le r^2/8.\label{Sbound2}
\ee
Setting $R=\min\{r,\rho\}$ we therefore have
\be
S_{D}(t)\chi_{R}\ge  \kappa \chi_{R},\qquad \forall t\le R^2/8.\label{Sbound3}
\ee
Now let $\mu (t)$ denote the unique local solution of the ODE 
\bs
\dot\mu = \kappa\gamma f(\mu ),\qquad \mu(0)=0
 \es
which exists and is positive for $t\in (0,T^*]$ for some $T^*>0$, i.e. $\mu (t)= \int_0^t \kappa\gamma f(\mu (s) )\, \dee  s $. The existence of such a $\mu$   follows by virtue of $f$ failing to satisfy the Osgood condition (\ref{osgood}). Setting $v(x,t)=\mu (t)\chi_{R}(x)$ it is clear that $v\in L^{\infty}(Q_{T^*})$. Furthermore, for $t\le T':=\min\{T^*,R^2/8\}$ we also have that
\bs
\int_0^t S_{\B}(t-s)[qf(v(s))]\, \dee  s&=&\int_0^t S_{\B}(t-s)[qf(\mu (s)\chi_{R})]\, \dee  s\\
&=&\int_0^t S_{\B}(t-s)[q\chi_{R}f(\mu (s))]\, \dee  s \ \text{($f(0)=0$)}\\
&=&\int_0^t f(\mu (s))S_{\B}(t-s)[q\chi_{R}]\, \dee  s \ \text{($S_{\B}$ linear)}\\
&\ge &\int_0^t \gamma f(\mu (s))S_{\B}(t-s)[\chi_{R}]\, \dee  s \ \text{($S_{\B}$ monotone)}\\
&\ge &\int_0^t \gamma f(\mu (s))S_{D}(t-s)[\chi_{R}]\, \dee  s \ \text{(Lemma \ref{lem:Sorder})}\\
&\ge &\chi_{R} \int_0^t \gamma \kappa  f(\mu (s))\, \dee  s \ \text{(Lemma \ref{lem:SD})}\\
&=&\mu (t)\chi_{R}=v  \ \text{(by definition of $\mu$)}
\es
and so $v$ is a  generalised  subsolution of (P) on $[0,T^\prime]$.

It is easy to see that $w(x,t)=t$ is a classical supersolution of (P) on $[0,\tau ]$ for any $\tau >0$ satisfying $f(\tau )\|q\|_{\infty}\le 1$, which is clearly possible by (H4). Furthermore, since $d\mu/dt \to 0$ as $t\to 0^+$, $\tau$ may also be chosen so that $v\le w$ on $[0,\tau ]$. Standard monotone iteration arguments then guarantee the existence of a bounded generalised  solution $U$ of (P) on $[0,T]$ satisfying $v\le U\le w$, where $T:=\min\{\tau ,T^\prime\}$. The positivity of
$U$ in $Q_T$ then follows from the integral representation of $U$ and the  positivity of $K_{\B}$. Finally, the regularity of $U$ follows from Remark \ref{rem:classical} when $q$ and $f$ are H\"{o}lder continuous.

For $c$ of indefinite sign let $\sigma =\|c\|_{\infty}$ and set $\tilde{c}(x)=c(x)-\sigma\le 0$ and $\tilde{f}(u)=f(u)+\sigma u$. The assumptions (H1) and (H4) are then satisfied with $c$ replaced by $\tilde{c}$ and $f$ replaced by $\tilde{f}$. Moreover, since   $\tilde{f}\ge f$, $\tilde{f}$ fails to satisfy
 the  Osgood condition (\ref{osgood}).
Hence, arguing as above, there exists a bounded, positive solution $\tilde{U}$ of the problem (P) with $\LO$ and $f$ replaced by $\tilde{\LO}:=\LO -\sigma$ and $\tilde{f}$, respectively. Clearly however, $\tilde{U}$ is a solution of  problem (P) with data $\LO$ and $f$ if and only if $\tilde{U}$ is a solution of  problem (P) with data $\tilde{\LO}$ and $\tilde{f}$, yielding the result.
\end{proof}

We can now combine Theorem~\ref{thm:possoln} and \cite[Theorem 1.4]{FW} to obtain the following  characterisation of uniqueness for (P). We point out that Theorem~\ref{thm:possoln} and Corollary~\ref{cor:main} below remain valid if hypothesis (H4) is replaced by a local one near zero, i.e., for
 $f:[0,M]\to [0,\infty )$.

\begin{corollary}
If {\rm (H1)-(H4)} hold then the following are equivalent:
\bi
\item[{\rm (i)}] $u=0$ is the unique bounded generalised  solution of the PDE {\rm (P)};
\item[{\rm (ii)}] $u=0$ is the unique non-negative solution of the ODE (\ref{ODE});
\item[{\rm (iii)}] $\displaystyle{\int_0^\eps \frac{\dee u}{f(u)}  =\infty}$ for some $\eps >0$.
\ei
\label{cor:main}
\end{corollary}

\section{Concluding Remarks}
We have obtained a simple necessary and sufficient condition on $f$ for uniqueness of the trivial solution in a semilinear parabolic equation
 with continuous, increasing nonlinearity $f$.  There were several key structural properties required to achieve this: (a) monotonicity of $f$;
 (b) semilinearity of the governing evolution equation; (c) monotonicity of the semigroup $S_\B$ (equivalently the kernel $K_\B$) with respect to the boundary data and its action on the underlying phase space $L^{\infty}(\Om )$ and (d) a lower bound on the action of the Dirichlet semigroup $S_D$ on characteristic functions (arising from a Gaussian lower bound on the Dirichlet kernel $K_D$). It seems reasonable to suggest that other evolution problems of the form
\bs
u^\prime = Au+f(u)
\es
having properties (a-d) would be amenable to the method employed here. For example, the fractional Laplacian $Au=-\left( -\Delta \right)^{s}$, $0<s <1$, would seem just such a case.

It would also be interesting to see if our method could be extended to continuous but non-monotone $f$ or to quasilinear operators such as the $p$-Laplacian $A(u)=\div \left(|\nabla u|^{p-2}\nabla u\right)$ ($p>1$) or  the porous medium operator $A(u)=\div \left( u^m\nabla u\right)$ ($m>0$). Whilst no integral equation formulation such as in (\ref{eq:integral}) or (\ref{eq:bounded}) exists in the quasilinear case it may still be possible to obtain non-uniqueness results via a weak formulation since monotonicity properties still hold in the weak sense (recalling that our subsolution in the proof of Theorem~\ref{thm:possoln} is not a classical subsolution, lacking as it does sufficient regularity).

\bibliographystyle{model1num-names}
\bibliography{<your-bib-database>}

\end{document}